\documentclass[11pt]{amsart}

\usepackage{amssymb,amsthm,amsfonts,amsmath}
\usepackage{a4wide}
\usepackage{amsrefs}
\usepackage[utf8]{inputenc}
\usepackage[T1]{fontenc}
\usepackage{xcolor}
\usepackage[colorlinks]{hyperref}
 \usepackage{ulem}
 \usepackage{graphicx}

\newcommand{\Z}{{\mathbb Z}}
\newcommand{\N}{{\mathbb N}}
 \newcommand{\G}{\mathcal{G}}
  \newcommand{\D}{\mathcal{D}}
\newcommand{\R}{{\mathbb R}}

\renewcommand{\phi}{{\varphi}}

\newcommand{\Per}{\mathrm{Per}}

\newtheorem{theorem}{Theorem}[section]

\newtheorem{lemma}[theorem]{Lemma}
\newtheorem{proposition}[theorem]{Proposition}

\newtheorem*{example}{Example}
\theoremstyle{definition}
\newtheorem{definition}[theorem]{Definition}
\newtheorem{remark}[theorem]{Remark}

\newcommand{\eps}{\varepsilon}

\newcommand{\EEE}{\color{black}}

\begin{document}
 
\title{Lattice tilings minimizing nonlocal perimeters}

\author{Annalisa Cesaroni}
\address{Dipartimento di Matematica "Tullio Levi-Civita", Universit\`{a} di Padova, Via Trieste 63, 35131 Padova, Italy}
\email{annalisa.cesaroni@unipd.it}
\author{Ilaria Fragal\`a}
\address{Dipartimento di Matematica, Politecnico di Milano, Piazza Leonardo da Vinci 32, 20133 Milano, Italy}
\email{ilaria.fragala@polimi.it}
\author{Matteo Novaga}
\address{Dipartimento di Matematica, Universit\`{a} di Pisa, Largo Bruno Pontecorvo 5, 56127 Pisa, Italy}
\email{matteo.novaga@unipi.it}

\begin{abstract} 
We prove the existence of periodic tessellations of $\R^N$  minimizing  a general nonlocal perimeter functional, defined as the interaction between a set and its complement through a nonnegative kernel, which we assume to be either  integrable at the origin, or singular, with a  fractional type singularity. We reformulate the optimal partition problem as an isoperimetric problem among fundamental domains associated with  discrete subgroups of $\R^N$, and we provide the  existence of a  solution  by using suitable  concentrated compactness type arguments and compactness results for lattices. Finally, we discuss  the possible optimality of the hexagonal tessellation  in the planar case. 
 \end{abstract} 
  
 
 
\maketitle 
 
\tableofcontents
 
\section{Introduction} 
An interesting class of problems in geometry and analysis concerns finding a tiling or  tessellation of $\R^N$, 
which is optimal for a given energy functional.  A famous problem in this direction, which is related to our work,  is the so-called {\it Kelvin problem},  posed by Lord Kelvin in 1887, see \cites{thompson, kelvin}: it consists in finding a partition of $\R^N$ into cells of equal volume, so that the total area of the surfaces separating them is as small as possible.  A complete solution to such problem is currently known only in dimension $N=2$, thanks to the celebrated paper by Hales  \cite{hales} 
(see also \cite{morgansolo}).   

We recall that a  tiling or tessellation of $\R^N$  is a collection of measurable subsets $\{E_k\}_{k\in\mathbb I}$,  where $\mathbb I$ is either finite or countable, such that 
\begin{enumerate} 
\item $|E_k|>0$ for all $k$,
\item $|E_k\cap E_j|=0$ for all $k\ne j$,
\item $|\R^N\setminus\cup_k E_k|=0$.
\end{enumerate}

In this paper we look for tessellations minimizing a general nonlocal perimeter functional, and we restrict to the class of   {\it lattice tilings},  that is, tessellations which are invariant  with respect to the action of  a $N$-dimensional group of translations (i.e.\ a lattice) of $\R^N$.

Such problem can be restated in a completely equivalent way by using the terminology of fundamental domains associated  with   a lattice. In particular,   each  fundamental domain $D$ for a  lattice $G$  is uniquely associated  with  a 
lattice  tiling of $\R ^N$, given by  $\{ D+g \} _ {g \in G}$. 
 

Let now describe the  criterium  for optimality among all possible lattice tilings, or equivalently among all fundamental domains. We consider a  
  measurable  kernel $K$ on $\R ^N$,  that we assume to be nonnegative and symmetric (i.e.\ satisfying  $K(x)=K(-x)\ge 0$ for every $x\in \R ^N$), 
 and we associate  with  $K$ the nonlocal perimeter
  \begin{equation}\label{perfrac}
  \Per_K(E):=  \int_{E }\int_{\R^N\setminus E}    K( x-y)dxdy \,.  
  \end{equation}
 More specifically, we  focus our attention on two relevant classes of kernels, respectively of  fractional and integrable type, i.e.\ satisfying one of the following assumptions
  \begin{equation}\label{frac}  \min \{ 1, |h| \} K ( h) \in L^ 1 (\R ^N) \, , \quad K(h)\geq C |h|^{-N-s} \text{ for some $C>0$ and $s\in (0,1)$;}
  \end{equation} 
  \begin{equation}\label{int}  K \in L^1(\R^N) \ \text{ with } \|K\|_1>0, \  \text{ and } \  \liminf_{z\to 0^+}\big [  K(z)-K(z+x) \big ] >0\ \  \forall x \neq 0.  
\end{equation}

The first tessellation problem we investigate  is obtained by fixing a discrete  group of translations. It can be stated precisely 
as follows: given a lattice $G$ in $\R^N$, 
we 
consider the isoperimetric problem
 \begin{equation}\label{iso1} \inf_{D\in\D_G} \Per_K(D)\,, 
 \end{equation}  
where 
$\D_G$ denotes the class of  fundamental domains for $G$, see Definition \ref{def:fund}
(notice in particular that all  fundamental domains for $G$ have the same volume). 

We prove the following existence result: 
 \begin{theorem}\label{thex1} 
 Let $K$ be a nonnegative symmetric  kernel on $\R ^N$,   satisfying one of the assumptions \eqref{frac}-\eqref{int}. 
 For any given lattice $G$ in $\R ^N$, problem \eqref{iso1} admits a solution, namely there exists a fundamental domain $D\in\D_G$ such that
 \[\Per_K(D)=\min_{E\in\D_G} \Per_K(E). \]
\end{theorem} 

 This result is obtained by  different methods  depending on whether  the interaction kernel is singular at the origin or not. 
 
 In the first case, we exploit the lower semicontinuity and compactness properties  of the perimeter functional, along with a by now  well-established  version of the concentration-compactness argument. For the local perimeter functional, this argument dates back to Almgren (see the monograph \cite{maggi}),  
and it has been already used to attack isoperimetric problems in finite or infinite clusters,  for local isotropic and anisotropic perimeters, see \cite{cnparti, cntiling, npst}. 
Let us also mention that  isoperimetric problems for the local perimeter  in the class of fundamental domains 
on compact Riemannian  manifolds have been considered in the literature, see \cite{choe, mnpr}.

 In the other case, when the interaction kernel is integrable at the origin, sets of bounded perimeter do not enjoy anymore any compactness property
  (see \cite{JW}),   so that
we have to adopt another strategy, which consists in relaxing the problem to $L^1$ densities. 
In this relaxed framework, a concentration-compactness argument  allows to 
prevent the complete loss of mass at infinity of a minimizing sequence, so that we obtain the  existence of an optimal density,  which minimizes 
a suitably relaxed perimeter functional.  Then, we derive the optimality conditions  satisfied by such an optimal density, from which    we deduce, 
by appropriate potential type arguments, qualitative properties of an optimal fundamental density. 
In particular, from the first variation, we deduce that an optimal fundamental density has bounded support. 
From the second variation, using the second condition in \eqref{int}, that is $\liminf_{z\to 0^+}\big [  K(z)-K(z+x) \big ] >0$ for $x\neq 0$, we conclude that the  
an optimal density is  actually the characteristic function of an optimal fundamental domain.  
Notice that, in general, if the interaction kernel is locally constant around the origin, we can prove only the existence of an optimal fundamental density, 
but we cannot expect it to be necessarily a characteristic function.   
Such arguments have been already exploited in the study of isoperimetric problems for nonlocal perimeter with one or two phases in \cite{cn,cdnp}. 
 
The second and more general isoperimetric problem we investigate is obtained by allowing the group $G$ to vary in the family $\G_m$ of lattices with volume $m$ (see Definition \ref{def:lattice}), hence admitting all possible lattice tilings of the space. More precisely, we consider the problem
 \begin{equation}\label{iso2} 
 \inf_{G\in\G_m} \inf_{E\in\D_G} \Per_K(E)\,. 
 \end{equation}  

Also for problem \eqref{iso2}, we obtain the existence of a solution:
 
 \begin{theorem}\label{thex2} 
  Let $K$ be a nonnegative symmetric  kernel on $\R ^N$,   satisfying one of the assumptions \eqref{frac}-\eqref{int}. 
 For any given $m>0$, problem \eqref{iso2} admits a solution, namely  there exists a lattice $G\in \G_m$ and a fundamental domain $D\in\D_G$ such that
 \[\Per_K(D)=\min_{G\in\G_m,\,E\in\D_G} \Per_K(E). \]
 \end{theorem} 
 
 This result is proved by combining the existence of  isoperimetric fundamental domains associated with a fixed group $G$, 
 with a nondegeneracy  property of a sequence of lattices with fixed volume, and with an
upper bound on the nonlocal  perimeter of the associated  fundamental domains. 
The nondegeneracy property mentioned above is based on classical compactness results for lattices, such as 
Mahler’s compactness Theorem and its generalizations (see \cite{ma, CS}). 
 
A natural question concerns the regularity of optimal fundamental domains in problems \eqref{iso1} and \eqref{iso2}. In this respect, the first observation is that every bounded fundamental domain has at least one singular point, due to the fact that its translations tessellate the space (for details we refer to Corollary 4.11 in \cite{cnparti}). 
In the case of integrable kernels,   that is,   when $K$ satisfies assumption \eqref{int},  
by potential theory arguments we obtain 
that every   optimal fundamental domain  is bounded (see Remark \ref{r:bounded}), but we cannot show partial regularity of the boundary.
On the other hand, when $K(h)=C |h|^{-N-s}$ for some $C>0$ and $s\in (0,1)$, by geometric measure theory arguments, 
it is possible to show  that  every optimal fundamental domain $D$ is bounded  and that $\partial D$ is smooth outside a closed nonempty singular  set $\Sigma$ with $\mathcal{H}^{N-1}(\Sigma)=0$ (discrete if $N=2$).  Such results   for  minimal  lattice tilings in the fractional setting are detailed in  \cite[Theorem 4.10]{cnparti} and are based on previous regularity results obtained for   finite partitions minimizing the fractional perimeter in \cite{colombomaggi}. 
 
Besides the  existence and regularity of minimal tessellations, a further 
interesting question is  the description of the   explicit structure of these partitions. This problem is quite hard, and, as mentioned at the beginning of this Introduction, 
the explicit shape of an optimal fundamental domain is known only in the planar case for local perimeters, see \cite{hales}. 
By analogy one expects that, also for nonlocal perimeters,   the optimal tiling  in the plane  should be the tessellation by regular hexagons.  
In the last section of the paper we give a partial result in this direction, holding when the class of admissible fundamental domains is restricted  to convex polygons.
 
 \smallskip
 
The paper is organized as follows. Section $2$ contains the main definitions and results about lattices in $\R^N$ we are going to use in the paper. Section $3$ is devoted to the proofs of Theorems \ref{thex1} and \ref{thex2} in the case $K$ is singular at the origin with fractional type singularity. Section $4$ deals with the proofs of the same results, in the case $K$ is integrable at the origin. Finally in Section $5$ we provide a partial result about the explicit shape of an optimal 
lattice tiling for nonlocal perimeters in the planar case. 
 
 \section{Lattices, fundamental domains and fundamental densities} 
We recall some standard definitions and basic properties of lattices in $\R^N$. 
For a general introduction to the subject we refer to the monograph \cite{CS}, and references therein. 
 
\begin{definition}\label{def:lattice}
A {\it lattice} is a discrete subgroup $G$ of $(\R^N, +)$ of rank $N$. 
The elements of $G$ can be expressed as $\sum _{i=1} ^N k_iv_i$, for a given basis  $(v_1, \dots, v_N)$ of $\R^N$, with    coefficients $k_i\in \Z$.   Any two bases for a lattice $G$ are related  by a matrix with integer coefficients and determinant equal to $\pm1$.  
 \end{definition}
 
 \begin{definition}
 The {\it volume of a lattice }$G$ is the number  $d(G)\in (0,+\infty)$ uniquely determined as 
 the absolute value of the determinant of the matrix of any set of generators $(v_1, \dots, v_N)$. Equivalently, if the lattice is viewed as a  discrete group of  isometries of $\R^N$, $d(G)$ coincides with the volume of the quotient torus $\R^N/G$. 
 
For a given $m\in (0,+\infty)$, we denote by $\G_m$ the family of  lattices  $G$ with $d(G)=m$.
\end{definition} 

\begin{definition}\label{def:dist} 
The  {\it minimum distance}   $\lambda(G)>0$  in a lattice $G$ is the length of the shortest nonzero element of $G$.
In particular, for every $p,q\in G$, there holds that $|p-q|\geq \lambda (G)$.
\end{definition}

\begin{definition}\label{def:fund} We say that $D\subset\R^N$ is a {\it fundamental domain} for a lattice $G$ 
if it  contains almost all representatives for the orbits of $G$ and the set of points
whose orbit has more than one representative has measure zero, i.e.
$$|\R ^N \setminus \bigcup _{g \in G} ( D + g) | = 0 \qquad \text{ and } \qquad  | D \cap ( D + g )| = 0 \quad \forall g \in G \setminus \{ 0 \} \,.$$ 

We denote by $\D_G$ the set of all fundamental domains of $\R^N$  for the group $G$.  

Notice that $|D|=d ( G)$ for all $D\in\D_G$. 
  \end{definition} 
 
\begin{example}
\upshape \label{remcomp} 
(i) If $(v_1, \dots, v_N)$ is a set of generators of  a lattice 
 $G$ as a $\Z$-module, then  the set $D=\{x=\sum_{i=1}^N t_iv_i, \ t_i\in [0,1)\}$ is a fundamental domain for $G$. 
 
 \smallskip
(ii) The {\it Voronoi cell} $V_G$  of a lattice $G$ is a fundamental domain for $G$.  Recall that $V_G$ is a  the centrally symmetric  convex polytope 
defined by  
 \[V_G:= \big \{x\in \R^N :\  |x|\leq |x-g| \quad \forall g\in G, g\neq 0 \big \}. \]

 \end{example}

We now introduce  a weaker version of the notion of fundamental domain, which will be useful when dealing with functionals with integrable interaction kernels. 
  \begin{definition}\label{def:fundrel} Given $G \in \mathcal G _m$, we call  {\it fundamental density} for the lattice $G$ 
a function $f$ in  the class 
  \begin{equation}\label{ag} \mathcal{A}_G :=\Big\{f\in L^1(\R^N; [0,1])  \ :\  \sum_{g\in G} f(x+g)= 1 \text{ a.e.} \Big\} \,.\end{equation}
   \end{definition} 

\begin{remark} Notice that,
when $f = \chi _D$, 
  $D$ is a fundamental domain for $G$ according to Definition \ref{def:fund} if and only if $f$ is a fundamental density for $G$ according to Definition \ref{def:fundrel}. 
  \end{remark}

\medskip
We now give a closer look at the properties of fundamental densities, specifically concerning their mass. The next lemma establishes in particular that all fundamental densities for a group $G \in \mathcal G_m$ have mass $m$ (item (ii)); moreover, for functions of mass $m$, the equality in the definition of fundamental density can be relaxed into an inequality  (item  (iii)).

\begin{lemma}\label{lemma1} 
Given $G \in \mathcal G _m$ and $ f\in L^1(\R^N; [0,1])$, we have: 
\begin{itemize}

\item[(i) ] If $\sum_{g\in G} f(x+g)\leq 1$ a.e., then $\int_{\R^N} f(x)dx\leq m$. 

\smallskip
\item[(ii)] If  $\sum_{g\in G} f(x+g)=1$ a.e., then $\int_{\R^N} f(x)dx= m$. 

\smallskip
\item[(iii)] If    $\sum_{g\in G} f(x+g)\leq 1$ a.e. and $\int_{\R^N} f(x)dx= m$, then  $\sum_{g\in G} f(x+g)=1$ a.e.
 
\end{itemize}
 \end{lemma}
\begin{proof} (i)  Let $D$ be a fundamental domain for $G$. By monotone convergence theorem,  we have
\[  \int_{\R^N} f(x)dx=\sum_{g\in G} \int_{D-g}  f (x)dx=\sum_{g\in G} \int_{D} f(x+g)dx=\int_{D} \sum_{g\in G} f(x+g)dx\leq \int_{D}1dx= |D|=m. \]

(ii)  Just replace the inequality in the above line by an equality. 

(iii) Assume by contradiction that   there exist $\eps>0$ and a measurable set $E$ with $|E|>0$ such that  $\sum_{g\in G} f(x+g)<1-\eps$ for a.e. $x\in E$. We may assume without loss of generality that $|E\cap( E+g)|=0$ for all $g\in G$, $g\neq 0$: it is sufficient to substitute $E$ with a subset   containing at most one representative for every orbit of $G$ (that is $E/G$). 
 
Consider the function $\tilde f(x):= f(x)+\eps \chi_E(x)$. It  belongs to $ L^1(\R^N; [0,1])$, with 
$\int_{\R^N}\tilde f(x)dx= \int_{\R^N} f(x)dx+ \eps |E|> m$.  Then, thanks to 
item (i), in order to get a contradiction it is enough to show that  the inequality $\sum_{g\in G} \tilde f(x+g)\leq 1$ holds true a.e. This is readily checked by construction. 
Indeed, 
for $x\not \in E$, we have  
by definition  that   $\sum_{g\in G} \tilde f(x+g)=\sum_{g\in G} f(x+g)$,  whereas,  for $x\in E$, we have $\chi_E(x+g)=0$ for all $g\in G$, $g\neq 0$, and hence 
$\sum_{g\in G} \tilde f(x+g)\leq \sum_{g\in G} f(x+g)+\eps\leq 1-\eps+\eps=1$.   \end{proof}

%
%
  We conclude this section by introducing a notion of convergence for lattices  \cites{cassels, ma}, and stating a  compactness result for lattices of given volume. 
%

 \begin{definition}\label{def:kur}
A sequence of lattices {\it $G_h$ converges in the Kuratowski sense to $G$}, if
\[G=\{g\in \R^N \ :\ \limsup_{h\to +\infty} d(g, G_h)=0\},\]
where $d ( \cdot,  G _h)$ denotes the Euclidean distance from $G _h$.  
\end{definition}

 \begin{remark}\label{r:lambda} 
(i)  The above definition is equivalent to the local Hausdorff convergence, namely $G_h\to G$ in the Kuratowski sense if and only if, 
for every compact set $K$ and for every 
$\eps>0$,  the following inclusions hold for $h$ sufficiently large: 
\[G\cap K\subseteq \{x\in K \ :\ d(x, G_h)\leq \eps\}\quad\text{and }\quad G_h\cap K\subseteq \{x\in K \ :\ d(x, G)\leq \eps\}.\]

\noindent (ii) The limit set $G$ in the Definition \ref{def:kur} turns out to be a closed subgroup of $(\R^N, +)$.

\noindent   (iii) The minimum distance functional $\lambda ( \cdot)$ introduced in Definition \ref{def:dist} is continuous with respect  to the Kuratowski 
convergence of lattices (when the definition is extended also to closed subgroups of $(\R^N, +)$, as the infimum of the length of their nonzero elements). 
 \end{remark}

\begin{theorem}\label{lemmaconvergence} 
Let $G_h$ be a sequence of lattices with $d(G_h)=m>0$ for every $h$.
Then, up to subsequences, $G_h\to G$ in the Kuratowski sense, where $G$ is either   a lattice with $d(G)=m$   or
a closed group which contains a line. In the second case, every sequence  $f_h$    of fundamental  densities  for $G_h$ converges to $0$  
weakly$^*$  in $L^\infty(\R ^N)$.  \end{theorem} 

\begin{proof}
First of all, by  local compactness of the Hausdorff metric, there exists a (not relabeled) subsequence $G_{h}$ and a closed group  
$G$ such that $G_{h}\to G$ in the Kuratowski sense.

%

If there exists a positive constant $\delta>0$ such that $\lambda(G_{h})\geq \delta>0$ for every $h$, then by   the compactness theorem for lattices due to Mahler \cite[Theorem 2]{ma}, possibly passing to a subsequence, we have that $G_h\to G$, where $G$ is a lattice with $\lambda(G)\geq \delta$. 
We recall from \cite[Theorem 1]{ma} that there exists a dimensional constant $C_N$ such that every lattice $G_h$  admits a set of generators 
$v_1^h, \dots, v_N^h$  with 
$\Pi_{i=1}^N |v_i^h|\leq C_N m$, and the same holds for $G$, with generators $v_1, \dots, v_N$.   
We infer that the fundamental domains $D_h=\{x=\sum_{i=1}^N t_iv_i^k, \ t_i\in [0,1)\}$  for $G_h$ 
converges in $L^1(\R^N)$ to the fundamental domain $D=\{x=\sum_{i=1}^N t_iv_i, \ t_i\in [0,1)\}$  for $G$, which implies that $d(G)=m$.   

If,  on the other hand, $\limsup_h\lambda(G_{h})=0 $, we claim that $G$ contains a line. 
Indeed, in this case,  by Remark \ref{r:lambda} (iii) we have that  $\lambda ( G) = 0$. Therefore,  $G$ cannot be discrete and hence  it contains a sequence 
 $g_i$ converging to some $g\in G$, with $g_i\ne g$ for all $i$
 and such that $\frac{g_i-g}{|g_i-g|}\to e$, where $e$ is a unitary vector of $\R^N$. 
 Fix $r\in \R$ and, for every index $i$, let $z_i\in\Z$ be such that 
 $\left| z_i |g_i-g|-r\right|\leq |g_i-g|$. Then we have 
 \[|z_i (g_i-g)- r e|\leq |g_i-g|+ |r| \left|\frac{g_i-g}{|g_i-g|}-e\right|,
 \] 
so that the line $\R e$ lies in the closure of $G$. 

Finally, let $f_h$ be a sequence of fundamental densities for $G_h$. Then, up to passing to a subsequence, $f_h$ converges to $f$ weakly$^*$ in 
$L^\infty (\R^N)$. It remains to show that $f=0$. 
It is sufficient to show that, 
still denoting by
 $\R e$ the line lying in the closure of $G$, it holds
\begin{equation}\label{f:linea} \sum_{n=1}^k f(x+t_n e)\leq 1 \qquad \text{ for a.e. } x \in \R^N \, , \ \forall k \in \N\, , \ \forall t_n \in (0, 1) \,.
\end{equation}   
Indeed, by applying the above inequality  to an infinitesimal sequence $t_n$, and passing to the limit as $n \to + \infty$, 
we obtain that $k f(x)\leq 1$ for a.e. $x\in \R^N$ and  every $k\in\N$, so that $ f= 0$ a.e.\ on $\R^N$. 
In order to prove \eqref{f:linea},   for every $n$ we  pick a sequence  $g_n^h\in G_h$ such that $g_n^h\to t _n e $ as $h \to + \infty$ (which exists
by the Kuratowski convergence of $G_h$ to $G$): since $f_h$ are fundamental densities for $G_h$, we have 
$$\sum_{n=1}^k f_h(x+g_n^h)\leq 1 \qquad \text{ for a.e. } x \in \R^N \, , \ \forall k, n \in \N  \,,$$  
and   \eqref{f:linea} follows by passing to the weak$^*$ limit.  \EEE

   \end{proof}

\section{Proofs of Theorems \ref{thex1} and \ref{thex2} for fractional kernels}
 
  For fractional kernels satisfying \eqref{frac}, Theorem \ref{thex1} has been proved in \cite{cnparti}. 
Below we give the proof of Theorem \ref{thex2}. It is a straightforward consequence of the following  lemma.


\begin{lemma}[Concentration compactness - fractional kernels]\label{lemmacc}   Assume that $K$ satisfies \eqref{frac}. Let $G_h$ be a sequence of groups in $\mathcal G _m$, and let $D_h$ be a sequence of fundamental domains for $G_h$, such that 
\begin{equation}\label{unibound}\sup_h \Per_K(D_h) <+\infty \,.\end{equation}
Then, up to passing to a subsequence,  there exist a lattice $G \in \mathcal G _m$,  families $g_h^i\in G_h$ ($i\in\N, h\in\N$), with $|g^i_h-g_h^j|\to +\infty$ as $h\to +\infty$ for $i\neq j$, and a collection $(D^i)$ of measurable sets in $\R^N$, such that 
\begin{enumerate}
\item[(i)] $G_h$ converges in the Kuratowski sense to $G$;
\smallskip

\item[(ii)]  for every $i \in \N$, $D_h-g_h^i\to D^i$ locally in $L^1(\R^N)$ as $h\to +\infty$;
\smallskip

\item[(iii)] $\sum_i |D^i| =m$,  $|D^i \cap D ^j| = 0$ for $i \neq j$, and  $D := \cup _i D ^ i$ is a fundamental domain for $G$. 

\smallskip
\item[(iv)] $\Per_K(D) \leq \sum_i\Per_K(D^i)\leq  \liminf_h\Per_K(D_h)$.
\smallskip

\end{enumerate} 
\end{lemma}
 
\proof By Theorem \ref{lemmaconvergence}, up to passing to a subsequence we can assume that $G_h$ converges in the Kuratowski sense either to a lattice $G \in \mathcal G _m$, or to a closed group which contains a line.

 By \cite[Lemma 3.4]{cnparti}, 
since
$D_h$ are measurable sets with fixed measure satisfying \eqref{unibound}, up to passing to a subsequence,  there exist  $z_h^i\in \Z^N$ ($i\in\N, h\in\N$),  with $|z^i_h-z_h^j|\to +\infty$ as $h\to +\infty$ for $i\neq j$, and a family $(D^i)$ of measurable sets in $\R^N$ such that 
$D_h-z_h^i\to D^i$  in $L^1_{\rm loc}(\R^N)$ and $\sum_i |D^i| =m$.  
This rules out the possibility that $G$ is a closed group which contains a line (otherwise $D_h-z_h^i$ would converge in $L^1_{\rm loc}(\R ^N)$ to $\emptyset$ by Theorem \ref{lemmaconvergence}), and hence $G$ belongs to $\mathcal G _m$.

By Remark \ref{r:lambda} (iii), 
we have that $\lambda(G_h)\to \lambda(G)>0$. Choosing $\delta>0$ such that $\lambda(G), \lambda(G_h)\geq \delta$ for every $h$, by \cite[Theorem 1]{ma}  there exists a   set of generators $v^i_h$ of $G_h$ such  that $\Pi_{i=1}^N |v^i_h|\ \leq C_N m$ and  $|v^i_h|\geq \delta$ for all $i$. 
Therefore there holds  $|v_h^i|\leq C_N m \delta^{1-N}:=R(m,\delta)$ for every $i$ and $h$. 
Now we  observe that  for every $i$, we may choose $g_h^i\in G_h$, with $|g_h^i-z_h^i|\leq\frac{\sqrt{N} R(m, \delta)}{2}$, 
so that $|g_h^i-g_h^j|\to +\infty$ as $h\to +\infty$, for $i\neq j$. 
Eventually passing to a subsequence,
 we get that  $D_h-g_h^i\to D^i$  in $L^1_{\rm loc}(\R^N)$.

If by contradiction $|D^i\cap D^j|>0$ for some  $i\neq j$, 
for $h$ sufficiently large we would have  
$|(D_h  - g_h ^ i) \cap (D_h-g_h^j)|>0$, against the fact that $D _h$ are fundamental domains for $G_h$. 
In a similar way, we obtain that 
 \begin{equation}\label{Q}|(D^i+g)\cap D^j|=0\qquad \text{$\forall g\in G \setminus \{  0 \}$, $\forall i,j$  }\end{equation} 
 (otherwise, writing  $g=\lim w_h$, for some $w_h\in G_h \setminus \{  0 \}$, for $h$ sufficiently large we would have  $|(D_h-g_h^i+w_h)\cap (D_h-g_h^j)|>0$).

Condition \eqref{Q}, together with the fact that $|D^i \cap D ^j| = 0$ for $i \neq j$ and with the equality $\sum_i |D^i|=m$,  implies that $D = \cup_i D^i$ is a fundamental domain for $G$ (see \cite[Lemma 2.2]{cnparti}).

%

The first inequality in (iv) follows from the submodularity property of fractional perimeter (see \cite[eq. (2)]{cnparti}); for the second inequality we refer to the proof of \cite[Lemma 3.4]{cnparti},   which continues to work with the unique modification that the elements  $g_h ^ i$ belong to $G_h$ in place of a fixed group $G$. \qed 

\bigskip 
 
{\bf  Proof of Theorem \ref{thex2} for fractional kernels.} 
 We consider a minimizing sequence of lattices $G_h\in\G_m$ and of fundamental domains $D_h$ for $G_h$, so that
 $$\lim _{h \to + \infty} \Per _K ( D_h) = \inf_{G\in\G_m,\,E\in\D_G} \Per_K(E)\,.
 $$ 
Since $D_h$ satisfy condition \eqref{unibound},
  we infer that there exist $g_h^i\in G_h$, measurable sets in $D^i\subset\R^N$, and a lattice $G \in \mathcal G_m$ as in Lemma \ref{lemmacc}. By using statement (iv) in the same lemma, we infer that  the lattice $G$ and its fundamental domain $D = \cup _i D _i$ 
solve the isoperimetric problem \eqref{iso2}. 

\qed

\section{Proofs of Theorems \ref{thex1} and \ref{thex2} for integrable kernels} 

In this section we prove Theorems \ref{thex1}  and \ref{thex2} under the assumption \eqref{int}. 

As already mentioned in the Introduction, when dealing with integrable kernels sequences of sets with uniformly bounded nonlocal perimeter are no longer precompact in $L ^ 1 _{\rm loc} (\R ^N)$. Therefore, we are led to relax the isoperimetric problems \eqref{iso1} and \eqref{iso2} by replacing fundamental domains by fundamental densities as introduced in Definition \ref{def:fundrel}, and accordingly by extending in a natural way the nonlocal perimeter.

In turn, we shall need to further enlarge the class of fundamental densities to  a class of functions where we are able to obtain a suitable version of the concentration-compacntess principle. 

%
%
%
%
%
%
%
%

\subsection{Relaxed isoperimetric problems} \label{sec1}

Let us introduce the following relaxed versions of the isoperimetric problems \eqref{iso1} and \eqref{iso2}: for a fixed lattice $G\in \mathcal \G_m$, we consider 
\begin{equation}\label{reliso1} \inf_{f\in {\mathcal A}_G} \mathcal{P}_K(f)\,,
 \end{equation}  
and then, we minimize among all lattices in $\G_m$ the previous energy 
 \begin{equation}\label{reliso2} \inf_{G\in\G_m} \inf_{f\in {\mathcal A}_G} \mathcal{P}_K(f)\,,
 \end{equation}    
 Here $\mathcal A _G$ denotes the class of fundamental densites,  according to Definition \ref{def:fundrel} and  $\mathcal{P}_K$ denotes the natural extension of the nonlocal perimeter to $L ^1$ functions,  as follows. 

  \begin{definition}\label{def:relenergy}
Given $f\in L^1(\R^N; [0,1])$, we set 
\begin{equation}\label{relaxed} 
\mathcal{P}_K(f):=\int_{\R^N}\!\!\int_{\R^N} f(x)[1-f(y)] K(x-y)dxdy
\end{equation}  
\end{definition}

In view of Lemma \ref{lemma1}, we are led to consider a larger class of relaxed fundamental densities,  which offers more stability  under passage to the limit due to the weaker constraints. 
\begin{definition}\label{def:relfundrel} Given $G \in \mathcal G _m$, we call  {\it relaxed fundamental density} for $G$ 
a function $f$ in  the class 
 \begin{equation}\label{cg} \overline{\mathcal{A}}_G :=\Big\{f\in L^1(\R^N; [0,1]) \ : \   \sum_{g\in G} f(x+g)\leq 1 \text{ a.e.} \Big\} \,,\end{equation}
   \end{definition} 

A property of the class $\overline{\mathcal{A}}_G$ which will be useful to our purposes is stated in the next result. 

 \begin{lemma}\label{lemma3}  
 Given $G \in \mathcal G _m$, 
for every $f\in \overline {\mathcal{A}}_G$, there exists $\Psi \in\mathcal{A}_G$ such that $ \Psi \geq f$ a.e.
\end{lemma} 
\begin{proof}
Let  $ f\in \overline{ \mathcal{A}}_G\setminus \mathcal A_G$. By Lemma \ref{lemma1} (iii), necessarily it must be  $\int_{\R^N}  f(x)<m$.  Moreover by Lemma \ref{lemma1} (ii), there exists a subset $E\subseteq \R^N$ with $|E|>0$ such that  $\sum_{g\in G} f(x+g)<1$ a.e. in $E$. 
By  arguing as in the proof of Lemma \ref{lemma1} (iii), we get that there exists $\psi \in \overline{\mathcal{A}}_G$ with $\psi  \geq  f$ and $\int_{\R^N}\psi  (x) \, dx > \int_{\R^N}  f(x) \, dx$. 
 Then   the set \[\mathcal{F}:=\Big \{ \psi \in \overline{\mathcal{A}}_G\ \text{ such that }  \psi \geq  f\text{ a.e. and } \int_{\R^N}\psi (x) \,dx > \int_{\R^N}  f(x) \, dx \Big \}\]
 is nonemtpy. The set $\mathcal{F}$ is partially ordered (with respect to inequality a.e.), and it turns out to contain an upper bound for every chain.  Actually, if $\psi_n\subset \mathcal{F}$ is an increasing sequence, we have  $\psi=\sup_n \psi_n\in \mathcal{F}$. Indeed, it is immediate that  $\psi \geq  f$ and, by monotone convergence,  $\int_{\R^N} \psi (x)dx\leq m$. Moreover, for every $G' \subseteq G$ finite, there holds $\sum_{g\in G'} \psi (x+g)=\sum_{g\in G'}\lim_n \psi _n(x+g)=\lim_n\sum_{g\in G'}\psi _n(x+g)\leq 1$ a.e., so that 
 $ \sum_{g\in G}   \psi  (x+g) = \sup_{G' \subseteq G, G' \text{finite}} \sum_{g\in G'} \psi (x+g)\leq 1$ a.e.
By applying Zorn lemma, we infer that $\mathcal{F}$ contains a maximal element $ \Psi$, which necessarily  satisfies $\int_{\R^N} \Psi (x)dx=m$.   \end{proof}

 \begin{remark}\label{change} 
 Since $K\in L^1(\R^N)$, setting  
$$\mathcal{J}_K(f)= \int_{\R^N}\!\!\int_{\R^N} f(x)f(y) K(x-y)dxdy\,, 
$$
for every $f\in L^1(\R^N; [0,1])$  there holds $\mathcal{P}_K(f)= \|f\|_1\|K\|_1- \mathcal{J}_K(f)$,   and hence  for $G\in\G_m$ there holds
$$
\inf_{f\in {\mathcal A}_G} \mathcal{P}_K(f) 
     = m \|K\|_1  - \sup_{f\in {\mathcal A}_G} \mathcal{J}_K(f) \,.$$  
\end{remark}

\begin{remark}\label{remfund} 
Since $K$ is nonnegative, the pointwise inequality $\Psi \geq f$ in Lemma \ref{lemma3}  implies that  $\mathcal{J}_K(\Psi) \geq \mathcal{J}_K(f)$.  
 It follows that, 
for every $G \in \G _m$, 
$$ \sup_{f\in {\mathcal A}_G} \mathcal{J}_K(f)   =   \sup_{f\in\overline{\mathcal A}_G} \mathcal{J}_K (f)\,; $$ 
moreover,  if the supremum of  $\mathcal{J}_K$ over  $\overline{\mathcal A}_G $ is attained, the same holds true for the supremum of $\mathcal{J}_K$  over ${\mathcal A}_G$. \end{remark} 

\subsection{Existence of an optimal fundamental density}  
 
\begin{proposition}\label{p:density} 
Let $K$ satisfy \eqref{int}. For every  $G \in \mathcal G _m$, the maximization problem
\begin{equation}\label{relmax} \sup_{f\in{ \mathcal A_G}} \mathcal{J}_K (f)\,. 
\end{equation}
admits a solution. 
\end{proposition}

\proof  In view of Remark \ref{remfund} it is sufficient to show that the maximization problem \[  \sup_{f\in{\overline{\mathcal A}_G}} \mathcal{J}_K (f)\] admits a solution. 

Let $\{ f _h \}$ be a maximizing sequence. Without loss of generality, by Lemma \ref{lemma3} and Remark \ref{remfund},  we may assume $f_h\in \mathcal A_G$. Let $Q$ be the Voronoi cell of $G$.  For $g\in G$, denoting by  $Q^g:=g+Q$, we define
$$m_{h,g} := \int_{Q^g} f_h(x)dx\,.$$ 
Given $\eps>0$ small, we divide the family of cells $Q^g$ in two sub-families:  
\begin{eqnarray}
& \displaystyle I^ {\eps}_h:=\big \{g \in G \ :\ m_{h,g}\leq \eps \big \},\qquad A^{\eps}_h:=\cup_{g\in I^{\eps}_ h} Q^g
& \label{fam1} \\  \noalign{\bigskip} 
&  \displaystyle J^ {\eps} _h :=\big \{g\in G\ :\ m_{h,g}>\eps \big \},
\qquad E^{\eps}_{h}:=\cup_{g\in J^{\eps}_ h} Q^g. & \label{fam2}
\end{eqnarray}

Observe that, due to the fact that $\int_{\R^N} f_h(x)dx= m$, 
we have $\# J^{\eps}_h\leq \frac{m}{\eps}$.

Given $w_h^i, w_h^l\in J^{\eps}_h$, up to subsequences we have $|w_h^i-w_h^l|
\to  c_{il}\in\N\cup \{+\infty\}$ as $h\to +\infty$. 

We fix $i$ and consider the clusters defined by \begin{equation}\label{clu} \mathcal{Q}^{ i, \eps}_h= \bigcup \Big \{ Q^{w_{l}} \ : \ 
w_{l}\in J^{\eps}_ h,  \ c_{i l} <+\infty \Big \}. \end{equation}
Note that the total number  $H_\eps$ of such clusters is at most $ \frac{m}{\eps}$.  
 
By construction, we have 
\begin{equation}\label{clu1}{\rm dist}(\mathcal{Q}^{i,\eps}_h, \mathcal{Q}^{l,\eps}_h)\to +\infty  \text{ as }  h\to +\infty \, , \ \text{ for every }  i\neq l \,,  \end{equation}
and
\begin{equation}\label{clu2}{\rm diam} (\mathcal{Q}^{i,\eps}_h )\leq \sum_{l\in \{1, \dots H_\eps\}, c_{il}<\infty} (2c_{il}+2 d_G)\leq M_\eps,\end{equation}
where $d_G$  denotes the diameter of $Q$, and  $M_\eps$ does not depend on $h$.

Setting
$$f_h^{i,\eps}:=f_h \chi_{\mathcal{Q}^{i,\eps}_h} $$
and choosing $g_{h } ^ {i, \eps} \in G \cap{\mathcal{Q}^{i,\eps}_h} $
so that the support of $ f_h^{i,\eps}(\cdot+g_{h } ^ {i, \eps} )$ is contained in $B(0,M_\eps)$ for every $h$, 
up to subsequences  $ f_h^{i,\eps}(\cdot+g_{h } ^ {i, \eps} ) \rightharpoonup^* f ^ { i, \eps}$ in $L ^ \infty (\R ^N)$, with 
 $\|f_h^{i,\eps} \| _ 1 \to\|  f ^ { i, \eps} \|_1$. 
 
 We now observe that 
the families $\eps \mapsto J ^ {\eps} _h$ are monotone in $\eps$ (with respect to inclusion), 
and hence the same is true also for the clusters ${\mathcal{Q}^{i,\eps}_h}$. This implies that we can assume that
the elements $g_{h } ^ {i, \eps}$ are independent of $\eps$, so that in the sequel we denote them by $g_h ^i$. 

Now, 
also the functions 
$f_h^{i,\eps} (\cdot + g_h ^i)$ are monotone in $\eps$ (with respect to a.e. inequality), and
the same property is inherited by  their limits $f^{i,\eps}$.  
By monotone convergence we infer that, as $\eps \to 0$, the sequence $f^{i,\eps}$
converges to some function $f^i$  in $L^1 (\R^N)$.
   Let us check that the function $$f:= \sum _{i = 1} ^ H f ^ i\, , \qquad \text{ where } H:= 
\lim_{\eps \to 0 }  H _ { \eps}\in (0 , + \infty]$$  
belongs to $\overline { \mathcal A} _ G$. 
From the inequality $\sum _{i=1} ^ { H _ \eps}  f _h ^{i,\eps} \leq f _h$, we infer that
$$\sum _ {g\in G}  \sum _{i=1} ^ { H _\eps}  f _h ^ {i, \eps} (x+g) \leq 1 \qquad \text{ a.e. } $$
The same inequality remains true also replacing $f_h^{i, \eps}$ first by  $f_h^{i,\eps}(\cdot+g_{h} ^ { i, \eps})$ (since $g_{h} ^ { i, \eps}\in G$), 
then by $ f^{i,\eps}$
 (by the weak* lower semicontinuity of the $L ^ \infty$-norm), and finally by their  limits  $f ^ i$ (by $L ^1$ convergence). 
%
%
%
%
So, we conclude that $ f \in \overline{\mathcal A} _ G$. 

\smallskip 
Let us prove that
$$\mathcal J _K (f) \geq \sum _{i= 1}^ {H _\eps}   \mathcal J _K (f^i ) \geq \limsup _h \mathcal J _K ( f _h)\,.$$

 Since $K\in L^1(\R ^N)$,   also its symmetric decreasing rearrangement $K^\star$ belongs to $L^1 (\R^N)$. Hence we can choose $R=R(\eps)$ in such a way that 
\begin{equation}\label{choiceR}
\lim_{\eps\to 0}R(\eps)=+\infty \quad \text{ and }\quad
\lim_{\eps\to 0}R(\eps)^N\int_{B_\eps} K^\star(y)dy =0.
\end{equation}
We claim that, with this choice of $R$, we have
\begin{equation}\label{aeps} \int_{A^{\eps} _h}\int_{\R^N} f_h(x)f_h(y) K(x-y)dxdy\leq r(\eps)\quad \text{ with $r(\eps)\to 0$ as $\eps\to 0$, uniformly in $h$.}\end{equation}  
Indeed, we have 
\begin{equation}\label{dueadd}\begin{array}{ll} 
\displaystyle \int_{A^{\eps}_ h}\int_{\R^N} f_h(x)f_h(y) K(x-y)dxdy =
& \!\!\!\!\!\!\!\!  \!\!\!\!\!\!\!\! \displaystyle \sum_{w\in G, g\in I^{\eps} _h, |g-w|\leq R} \int_{Q^g}\int_{Q^w} f_h(x)f_h(y) K(x-y)dxdy + 
\\  \noalign{\bigskip}
& \!\!\!\!\!\!\!\!  \!\!\!\!\!\!\!\!   \displaystyle \sum_{w\in G, g \in I^{\eps}_ h, |g-w|>R} \int_{Q^g}\int_{Q^w} f_h(x)f_h(y) K(x-y)dxdy\,.
\end{array}
\end{equation}
For every $w\in G$  and $g\in I^{\eps}_h$, it holds 
\begin{multline}\int_{Q^g}\int_{Q^w} f_h(x)f_h(y) K(x-y)dxdy \leq \int_{B_{m_{h,w}}}\int_{B_{m_{h,g}}} K^\star(x-y)dxdy\\= \int_{B_{m_{h,w}}}\int_{B_{m_{h,g}}(x)} K^\star(y)dydx\leq m_{h,w}\int_{B_{m_{h,g}}} K^\star(y)dy\leq m_{h,w}\int_{B_\eps} K^\star(y)dy,\end{multline} where we have used in the order 
Riesz rearrangement inequality, a change of variable, 
the fact that $g\in I^{\eps}_h$, and the 
fact that $K^\star$ is symmetrically decreasing. 
Hence, we have the following bound for the first addendum in \eqref{dueadd}:
$$\begin{array}{ll} \displaystyle \sum_{w\in G, g\in I^{\eps} _ h, |g-w|\leq R} \int_{Q^g}\int_{Q^w} f_h(x)f_h(y) K(x-y)dxdy
& \displaystyle \leq \sum_{w\in G} m_{h,w}(2R)^N \int_{B_\eps} K^\star(y)dy \\ 
& \displaystyle \leq m\,  (2R)^N\!  \int_{B_\eps} K^\star(y)dy. \end{array}
$$ 
On the other hand, the second addendum in \eqref{dueadd} can be estimated as
$$\begin{array}{ll}  \displaystyle \sum_{w\in G, g\in I^{\eps}_h, |g-w|>R} \int_{Q^g}\int_{Q^w} f_h(x)f_h(y) 
K(x-y)dxdy 
&
\leq \displaystyle  \sum_{g\in I^ {\eps}  _h }m_{h, g}\int_{|y|>R-2d_G}    K(y) dy
\\ 
&\displaystyle  \leq 
m\int_{|y|>R-2d_G}  K(y) dy.\end{array}
$$ 
By combining the two above estimates, we obtain 
\begin{equation}\label{eqRR}
\int_{A^{\eps} _ h}\int_{\R^N} f_h(x)f_h(y) K(x-y)dxdy\\\le
m (2R)^N \int_{B_\eps} K^\star(y)dy + m\int_{|y|>R-2d_G}  K(y) dy, 
\end{equation}
so that  \eqref{aeps} holds true thanks to the choice of $R = R (\eps)$. 

 
Now, by using  \eqref{aeps} we have 
%
$$\begin{array}{ll} \limsup_h\mathcal{J}_K(f_h) &\displaystyle \leq \limsup_h  \int_{E^{\eps} _ h }\int_{E^{\eps} _ h}f_h (x)f_h (y)K(x-y)dxdy + 2r(\eps)\\ \noalign{\medskip} 
& \displaystyle \leq  \sum_{i=1}^{H_\eps} \limsup_h  \int_{\R^N}\!\!\int_{\R^N}f_h^{i,\eps}(x+g_h^{i,\eps})f_h^{i,\eps}(y+g_h^{i,\eps}) K(x-y)dxdy+ 2r(\eps)\\ \noalign{\medskip} 
& \displaystyle =   \sum_{i=1}^{H_\eps} \int_{\R^N}\!\!\int_{\R^N}f^{i,\eps}(x)f^{i,\eps}(y)K(x-y)dxdy+2r(\eps)\,, 
\end{array}
$$ 
where the last equality follows from the fact that $\mathcal{J}_K$ is continuous with respect to the  tight  convergence (see for instance \cite{cdnp}). 

Finally, we pass to the limit as $\eps \to 0$. Thanks to the continuity of $\mathcal J _ K$ with respect to the
$L^1$-convergence,   we obtain
$$ \limsup_h\mathcal{J}_K(f_h)  
 \leq  \sum_{i=1}^{H } \int_{\R^N}\!\!\int_{\R^N}f^{i}(x)f^{i}(y)K(x-y)dxdy 
\leq \int_{\R^N}\!\!\int_{\R^N} f (x)  f (y)K(x-y)dxdy\,.
 $$ \qed 

  \subsection{Optimality conditions}
  
  In this section we derive first and second order optimality conditions satisfied by 
an optimal fundamental density given by Proposition \ref{p:density}. They will be exploited to prove that   optimal fundamental densities are in fact   characteristic functions of  optimal fundamental domains.
First of all,  we associate with any fixed function $f\in \mathcal{A}_G$, with $G\in \G_m$,   the following potential   
   \begin{equation}\label{pot}
V(x):=\int_{\R^N} f(y)K(x-y)dy. \end{equation} 
It turns out that 
\begin{itemize}
\item[(i)]  $V\in C(\R^N)\cap L^1(\R^N)\cap L^\infty(\R^N)$, with $0<V\leq \|K\|_{L^1}$ and $\|V\|_{L^1}=m\|K\|_{L^1}$.  
\item[(ii)] $\lim_{|x|\to +\infty}V(x)=0$. 
\item[(iii)]  $\sum_{g\in G} V(x+g)= \|K\|_{L^1}$  for all $x\in \R^N$. 
\end{itemize} 
 For items (i)-(ii), see \cite[Proposition 5.2]{cn}, while the last item is an immediate consequence of condition  $\sum_{g\in G} f(x+g)=1$. 
 
Let us also denote by $Q$ a bounded fundamental domain of $G$, by $Q'$ the set of Lebesgue points of $f$ 
in $Q$, and let us introduce the following sets:  
  \begin{equation}\label{sets}S=\{x\in Q' \ :\  f(x)=1\}\, , \quad N=\{x\in Q'\ :\  f(x)=0\}\, , \quad  D=Q'\setminus (S \cup N)\,. 
  \end{equation}
 
We point out  that,  in view of the equality $ \sum_{g\in G} f(x+g)=1$, we have
\begin{eqnarray}
x \in D  \ \Rightarrow\  \exists \,  g\in G \, , \  x + g \in D\,.&\label{sposto1} 
\\ 
x \in S  \ \Rightarrow\  \forall g \in G\, , \  x + g \in N\,.&\label{sposto2} 
\\ 
x \in N  \ \Rightarrow\  \exists\,  g \in G\, , \  x + g \not \in N\,.&\label{sposto3} 
\end{eqnarray} 

\begin{lemma}\label{lemmavar}
Let   $f\in\mathcal{A}_G$ be an optimal fundamental density given by Proposition \ref{p:density}. Associate with $f$ the potential  $V$ 
 as in  \eqref{pot}, and the sets  $D$, $N$, $S$ as in \eqref{sets}.  

(i) For every function  
$\psi$ such that  $f+\lambda \psi\in \mathcal{A}_G$ for $\lambda >0$ sufficiently small, it holds 
\begin{equation}\label{first}  \int_{\R^N} \psi(y)V(y)dy \leq 0\, . \end{equation}    
Consequently,  the potential $V$  satisfies
\begin{eqnarray}&  V(x)= V(x+g)  \qquad \forall x\in D\, , \ \forall  g \in G \  :\ x + g \in D\, ; \quad & \label{D}  
 \\ 
& V(x)\geq   V(x+g) \qquad \forall x\in S\, ,  \  \forall g\in G\, ;\, \qquad \qquad\qquad \quad& \label{S}  
\\
 &  V(x) \leq  V(x+g)  \qquad \forall x\in N\, , \  \forall g  \in G \ : \ x+g \not \in N\, . \quad & \label{N}  
 \end{eqnarray}

(ii)  For every function  
$\psi$ such that  $f+\lambda \psi\in \mathcal{A}_G$ for $\lambda\in (-1,1)$, $|\lambda|$ sufficiently small, it holds
\begin{equation}\label{second}  
 \int_{\R^{2N}} K(x-y)\psi(x)\psi(y)dxdy\leq 0  \,. 
\end{equation}

%
%
%
\end{lemma}

\proof By maximality of $f$, exploiting the symmetry of $K$ and the definition of $V$, we get  
 \begin{equation}\label{max}0\geq  \mathcal{J}_K(f+\lambda \psi)- \mathcal{J}_K(f)= 2\lambda\int_{\R^N} \psi(x)V(x)dx+ \lambda^2\int_{\R^{2N}} K(x-y)\psi(x)\psi(y)dxdy.\end{equation}
 which implies conditions \eqref{first} and \eqref{second}.

To deduce the pointwise relations \eqref{D}-\eqref{S}-\eqref{N} from condition \eqref{first},  we choose some particular functions $\psi$. 
 
Let $ x\in D$, and let $g\in G$ be such that $x + g \in D$ (which exists by \eqref{sposto1}). 
Then, for $r>0$ sufficiently small,  there exists a set $E_r(x)$ of positive measure with $x \in E _ r (x)\subset B _ r (x)$, such that the function $\psi (y) := \chi _{E_r (x)} - \chi_ {E _ r (x) + g}$ satisfies the condition $f + \lambda \psi \in \mathcal A _G$ for $|\lambda| $ sufficiently small.  Writing the inequality \eqref{first} for  $\pm \psi$, we infer that 
$$ \int_{E _ r (x)} [ V (y) - V ( y +g) ]\, dy  = 0\,,$$ 
which implies \eqref{D} dividing by $|E _ r ( x) |$ and letting $r \to 0 ^+$. 
%
 
Let  $x\in S$,  and let $g \in G$ be arbitrary. Taking into account \eqref{sposto2}, for $r>0$ sufficiently small 
there exists a set $E_r(x)$ of positive measure with $x \in E _ r (x)\subset B _ r (x)$, such that the function
$\psi (y) = - \chi _{E _ r (x) } + \chi_ {E _ r ( x )+ g}$ satisfies the condition $f + \lambda \psi \in \mathcal A _G$ for $\lambda >0$ sufficiently small. 
Writing the inequality \eqref{first} for  $\psi$, we infer that 
$$\int_{E _ r (x)} [- V (y) + V ( y +g) ]\, dy  \leq 0\,,$$ 
which implies \eqref{S} dividing by $|E _ r ( x) |$ and  letting $r \to 0 ^+$.


Let  $x\in N$,  and let $g \in G$ be such that $x + g \not \in N$ (which exists by \eqref{sposto3}).  
For $r>0$ sufficiently small, 
there exists a set $E_r(x)$ of positive measure with $x \in E _ r (x)\subset B _ r (x)$, such that the function 
$\psi (y) =  \chi _{E _ r (x) } - \chi_ {E _ r ( x) + g}$ satisfies the condition $f + \lambda \psi \in \mathcal A _G$ for $\lambda >0$ sufficiently small. 
Writing the inequality \eqref{first} for  $\psi$, we infer that 
$$\int_{E _ r (x)} [ V (y)  - V( y +g) ]\, dy  \leq 0\,,$$ 
which implies \eqref{N} dividing by $|E _ r ( x) |$  and  letting $r \to 0 ^+$.  \qed


 \bigskip

\subsection{Existence of an optimal fundamental domain} 
Using Proposition \ref{p:density} and Lemma \ref{lemmavar} we are ready to prove the 
existence of an optimal fundamental domain, among fundamental domains for a fixed lattice $G\in\G_m$. 

\begin{proof}[Proof of Theorem \ref{thex1}]
In view of Remark \ref{change}, it is enough to show that 
a solution  $f\in \mathcal{A}_G$ to the maximization problem \eqref{relmax}, which exists by Proposition \ref{p:density}, 
is necessarily the characteristic function of a bounded set $E$.  

We claim first of all that $f$ is compactly supported. We argue by contradiction. If $\text{supp} f$ is not bounded, there exists a sequence $x_h$ of Lebesgue points for $f$, with $|x_h|\to +\infty$,  such that $f(x_h)>0$. For every $h$, let $g_h\in G$ such that $x_h+g_h\in Q'$, where $Q'$ denotes as above the set of Lebesgue points of $f$ in a convex fundamental domain $Q$. Then, 
keeping the same notation as  in Lemma \ref{lemmavar} we have that,  for every $h$,   either 
$x_h+g_h\in D$ or $x_h+g_h\in S$.

In the first case, we get by \eqref{D} that $V(x_h)= V(x_h+g_h) $; in the second case, we get by \eqref{S}  that  $V(x_h)\geq V(x_h+g_h)$. In both cases, since $x _h + g _h \in Q$, it follows that $V(x_h)\geq v:=\min_{x\in \overline Q} V(x)>0$. By the strict positivity of $V$, we have that $v>0$. 
So $V ( x_h ) \geq v >0$, but this contradicts the fact that  $|x _h | \to+ \infty$,  since $\lim_{|x|\to +\infty} V(x)=0$.

Next let us prove that $f$ is the characteristic function of a set $E$. 
Assume by contradiction that $D\neq \emptyset$. 
For $ x\in D$, we proceed in a similar way as done  in the proof of Lemma \ref{lemmavar}. We select $g\in G$ such that $x + g \in D$ (which exists by \eqref{sposto1}). Then,  for $r>0$ sufficiently small,  
  we consider
a set $E_r(x)$ of positive measure, with $x \in E _ r (x)\subset B _ r (x)$, such that the function 
 $\psi (y) := \chi _{E _ r (x) } - \chi_ {E _ r(x)+ g}$, 
satisfies the condition  $f + \lambda \psi \in \mathcal A _G$ for $|\lambda|$ sufficiently small.  Writing the inequality
\eqref{second} for such a function $\psi$, we obtain 
\[ 
\int_{E_r(x)}\int_{E_r(x)} [K(t-y)-K(t-y+g)]  dydt\leq 0.
\]
This contradicts  the last condition on $K$ appearing in \eqref{int}. Indeed, such condition implies that,  for $r>0$ sufficiently small, there 
exists $\eps>0$ such that, for every $t,y\in B_r(x)$, since $|t-y|\leq 2r$,  it holds 
$K(t-y)-K(t-y+g)>\eps>0$.
 \end{proof}

\medskip
\begin{remark}\label{r:bounded} Note that the above proof shows in particular that, when $K$ satisfies assumption \eqref{int}, an optimal fundamental domain is necessarily bounded. \end{remark} 
\medskip

\subsection{Nondegeneracy of lattices}  
In order to pass to the minimization problem \eqref{iso2}, we need the following result,  which establishes that the possible degeneracy  of a sequences of lattices in $\mathcal G _m$ in the statement of Theorem \ref{lemmaconvergence} is ruled out 
whenever  a sequence of associated  fundamental densities has Riesz energy uniformly bounded away from $0$. 
\begin{lemma}\label{lemmacompattezzaint} 
Let $G_h$ be a sequence of lattices in $\mathcal G _m$. Assume that for every $h$  there exists $f_h\in \mathcal{A}_{G_h}$ such that 
\begin{equation} \label{condizione1}   \inf_h \mathcal J_K(f_h) >0.\end{equation}
Then, up to subsequences,  $G_h$ converges in the Kuratowski sense to  a lattice in $\mathcal{G}_m$.
\end{lemma}
\begin{proof} By Theorem \ref{lemmaconvergence}, up to subsequences $G_h\to G$ in the Kuratowski sense, where $G$ is either   a lattice  in $\mathcal{G}_m$  or a closed group which contains a line. Moreover, in order to exclude the second possibility, it is enough to show that $f_h$ cannot converge to $0$ weakly$*$ in $L^\infty  (\R ^N)$. 
We proceed in a similar way as in the proof of Proposition \ref{p:density}. 
 In place of choosing $Q$ as  a bounded fundamental domain associated with some lattice, we  choose 
$Q=[0,1]^N$ and, for $z\in \Z^N$, we set  
$$m_{h, z} := \int_{Q+z} f_h(x)dx\,. $$  
Then, given $\eps>0$ small, we define $I ^ \eps_h$ and $J ^ \eps _h$ as in \eqref{fam1} and \eqref{fam2}, with $\Z ^N$ in place of the lattice $G$, and
we divide  the family of cells $Q+z$ in the  two sub-families
$A ^ \eps _h$ and $E ^ \eps _h$, with $m _{h,z}$ in place of $m_{h, g}$. 
Again, due to the fact that $\int_{\R^N} f_h(x)dx= m$, 
we have $\# J^{\eps}_h\leq m/\eps$. 

Then we proceed by defining the clusters  $\mathcal{Q}^{ i, \eps}_h$ as  in \eqref{clu}. By construction, conditions \eqref{clu1} and \eqref{clu2} hold. 
Setting
\begin{equation}\label{fieps} f_h^{i,\eps}:=f_h \chi_{\mathcal{Q}^{i,\eps}_h} 
\end{equation}
and choosing $z_{h } ^ {i, \eps} \in G \cap{\mathcal{Q}^{i,\eps}_h} $
so that the support of $ f_h^{i,\eps}(\cdot+z_{h } ^ {i, \eps} )$ is contained into $B(0,M_\eps)$ for every $h$, 
up to subsequences,  $f_h^{i, \eps} (\cdot+z_h^{i,\eps})\stackrel{*}\rightharpoonup f^{i, \eps}$ weakly$^*$  in $L^\infty$ and  
$\| f_h^{i, \eps} \|_1 \to \| f^{i, \eps}  \|_1$.

By repeating the same estimates as in the proof of Proposition \ref{p:density}, we arrive at the inequality \eqref{aeps}, which in turn implies 
\begin{equation}\label{ene1} \mathcal{J}_K(f_h) \leq  \int_{E^{\eps} _ h }\int_{E^{\eps} _ h}f_h (x)f_h (y)K(x-y)dxdy + 2r(\eps)\end{equation} 
By combining \eqref{ene1} with assumption \eqref{condizione1}, we infer that, 
for $\eps>0$ sufficiently small,
 \[  \inf_h \int_{E^{\eps} _ h }\int_{E^{\eps} _ h}f_h (x)f_h (y)K(x-y)dxdy>0\,. \] 
Hence, we have
\[\inf_h  \sum_{i=1}^{H_\eps}\mathcal J _K ( f_h^{i, \eps} )  > \inf_h  \int_{E^{\eps} _ h }\int_{E^{\eps} _ h}f_h (x)f_h (y)K(x-y)dxdy>0\,. \] 
Finally, we invoke the  continuity of  $\mathcal J _ K$ with respect to the tight convergence, which implies 
\[  \sum_{i=1}^{H_\eps} \mathcal{J}_K( f^{ i, \eps}) = \lim_h  \sum_{i=1}^{H_\eps}\mathcal{J}_K( f_h^{ i, \eps}) >0\, . \] 
We conclude that that exists at least one index $i\in 1, \dots H_\eps$ such that $  \mathcal{J}_K( f^{i, \eps})>0$. 
So, since $f _h\geq f _h ^ {i, \eps}$, the sequence $f_h$ cannot converge to $0$ weakly$^*$ in  $L^\infty (\R ^N)$.  
\end{proof}

\subsection{Existence of a minimal lattice tiling} 

We conclude this section providing the solution to the minimization problem \eqref{iso2}. 
\begin{proof}[Proof of Theorem \ref{thex2}]


By Remark \ref{change},  solving the minimization problem \eqref{iso2} is equivalent to solving the maximization problem 
\[ \sup_{G\in \G_m,\ f\in {\mathcal A}_G} \mathcal{J}_K(f).\]We consider a maximizing sequence of lattices $G_h\in\G_m$ and of fundamental densities $f_h$ for $G_h$, so that
\[\lim _{h \to + \infty} \mathcal{J}_K ( f_h) = \sup_{G\in \G_m,\ f\in {\mathcal A}_G} \mathcal{J}_K(f)>0.\]
This implies in particular that condition \eqref{condizione1} holds. So, by Lemma \ref{lemmacompattezzaint}, 
up to passing to a subsequence,  $G_h$ converge in the Kuratowski sense to a lattice $G \in \mathcal G_m$. 
  
We proceed as in the proof of  Proposition \ref{p:density} and of Lemma \ref{lemmacompattezzaint}. We fix $Q=[0,1]^N$ and  we construct the clusters $\mathcal{Q}^{i,\eps}_h$, we define $f_h ^ {i, \eps}$ as in \eqref{fieps}, and we choose 
$z_h^{i, \eps}\in \mathcal{Q}^{i,\eps}_h\cap \Z^N $,  such that  up to subsequences,  
$f_h^{i, \eps} (\cdot+z_h^{i,\eps})\stackrel{*}\rightharpoonup f^{i, \eps}$ weakly$^*$  in $L^\infty(\R^N)$, with  
$\| f_h^{i, \eps}  \|_1 \to \| f^{i, \eps}  \|_1$.

By arguing as in the proof of Lemma \ref{lemmacc}, we can replace $z_h^{i,\eps}$ by $g_h^{i,\eps} \in G_h$ . 
Indeed, since $\lambda(G_h)\to \lambda(G)>0$,   we can choose $g_h^{i,\eps} \in G_h$, with $|g_h^{i,\eps}-z_h^{i,\eps}|$ uniformly bounded in $h$ and $\eps$, so that $|g_h^{i,\eps}-g_h^{j,\eps}|\to +\infty$ as $h\to +\infty$, for $i\neq j$. So, $f_h^{i, \eps} (\cdot+g_h^{i,\eps})\stackrel{*}\rightharpoonup f^{i, \eps}$ weakly$^*$  in $L^\infty (\R ^N)$.  

Moreover,  by arguing as in the proof of Proposition \ref{p:density}, we can pass to the limit also  as $\eps \to 0$: for every $i$,  the sequence $f^{i,\eps}$
admits a limit $f^i$  in $L^1 (\R^N)$, and the function 
$f:= \sum _{i = 1} ^ H f ^ i$  (with $H:= 
\lim_{\eps \to 0 }  H _ { \eps}\in (0 , + \infty]$) turns out to  
belong to $\overline { \mathcal A} _ G$. 
 
 Arguing as in Proposition \ref{p:density}, we arrive at the inequality \eqref{aeps}, which in turn implies 
\[\mathcal{J}_K(f_h) \leq  \int_{E^{\eps} _ h }\int_{E^{\eps} _ h}f_h (x)f_h (y)K(x-y)dxdy + 2r(\eps).\]

By using  the continuity of  $\mathcal J _ K$ with respect to the tight convergence,  passing to the limit as $h\to +\infty$, we obtain 
$$\lim_{h }\mathcal{J}_K(f_h) \leq
\lim_h  \sum_{i=1}^{H_\eps}\mathcal{J}_K( f_h^{ \eps,i}) +2r(\eps)= \sum_{i=1}^{H_\eps} \mathcal{J}_K( f^{ \eps,i})+ 2 r ( \eps).$$ 
By 
using the continuity of $\mathcal{J}_K$ with respect to $L^1$ convergence, as passing to the limit as $\eps\to 0$, we  obtain
$$\lim_{h }\mathcal{J}_K(f_h) \leq \mathcal J _K ( f) $$
and hence that 
 $f$ is a solution to the isoperimetric problem $\sup_{f\in {\overline {\mathcal{A}}_G}} \mathcal{J}_K(f)$.

Finally, the fact that  $f$ is the characteristic function of a  (bounded) fundamental domain $D$ of $G$ 
is  obtained by the same  proof as in Theorem \ref{thex1}. \end{proof}

\section{Optimal tilings in two dimensions} 
 
 A natural question suggested by our existence results is whether is it possible to determine explicitly the geometry of an optimal tiling, at least in dimension $N = 2$. 
 In case the cost is given by the classical perimeter, and in the more general framework  of cells having merely equal area but not necessarily identical shape, 
a celebrated theorem  by Hales \cite{hales} states that the optimal  configuration is given by the  tessellation with regular hexagons (see also  \cite{morgansolo}).  A similar behaviour has been recently established also in case of optimal partitions for the Cheeger constant \cite{BBFV, BF1} and for 
 Robin eigenvalues \cite{BF2}.   
 
By analogy, the tessellation with regular hexagons is  a natural candidate to be an optimal configuration also in the framework of nonlocal perimeters, at least for periodic foams  with equal cells. 
 Below we give a partial results in this direction, restricting further the class of admissible configurations. 
Specifically, we consider the following simplified variant of problem \eqref{iso2}: \begin{equation}\label{iso2poly} 
 \inf_{G\in\G_m} \inf \Big \{  \Per_K(E) \ :\ E \in \mathcal D _G, \ E \text { convex polygon} \Big \} \,.
 \end{equation}  

Let us mention that the nonlocal isoperimetric problem in the class of polygons  with a fixed number of sides has been recently considered in   \cite{BCT,BBF}. 
In particular, in \cite{BBF} it is shown that such problem is more delicate than expected, because, depending on the choice of the kernel, the phenomenon of symmetry breaking (meant as the non optimality of the regular gon) may occur for every even number of sides  larger than or equal to $6$. 

However, in \eqref{iso2poly} we deal only with convex polygons which tessellate the plane  by translations. Recall that   a convex polygon which tessellates the plane  by translations  can only be either a centrally symmetric hexagon or a parallelogram \cite{mm, fedorov}. 
Hence, admissible polygons in problem \eqref{iso2poly} are reduced to centrally symmetric convex hexagons,  possibly degenerated into a parallelogram,  that is, possibly having  two triples of vertices aligned.  For such polygons, the following simple symmetrization result holds:


\begin{lemma}\label{lemmahex} 
Let  $\mathcal H $ be the class of centrally symmetric convex hexagons of unit area, possibly degenerated into a parallelogram.   Then:

\begin{itemize}
\item[(i)] Every $H \in \mathcal H$ can be transformed  by two Steiner symmetrizations into  the regular hexagon $H^*$ of unit area. 

\item[(ii)]   For any nonnegative symmetric  kernel $K$ on $\R ^2$,   satisfying  \eqref{frac} or \eqref{int},    there holds 
$\min_{H\in \mathcal{H}} \Per_K(H) =  \Per_K(H^*)$. 
Moreover, if $K$ is strictly decreasing, then  $H^*$ is the unique minimizer.
 \end{itemize} \end{lemma} 
 
\begin{proof}
Let $H \in \mathcal H$ and let us denote by $A,B,C,D,E,F$ its vertices, ordered  in counterclockwise sense. We apply to  $H$ a first  Steiner symmetrization with respect of the axis of a diagonal
connecting the  vertices of two consecutive edges, say $AE$. 
The triangle $AEF$ becomes a isosceles triangle $AEF'$, the edges $AB$ and $DE$ become edges $AB'$, $ED'$ parallel to the symmetry axis, and finally the triangle $BDC$ becomes the isosceles triangle $B'D'C'$. 
In particular,  the pairs  $AB'$ and $ED'$, $AF'$ and $F'E$ and $B'C'$ and $C'D'$ are congruent, and by symmetry with respect to the origin, also the pairs $AF'$ and $C'D"$, 
$C'B'$ and $EF'$ are congruent. The polygon $H^*$ we obtained is again a hexagon, whose vertices are not aligned three by three 
(see Figure \ref{figura}). 
Now  we repeat the same argument  with respect to the axis of  the diagonal $F'D'$, thus obtaining a hexagon with equal sides, which is symmetric with respect to the origin and also with respect to the diagonals joining two opposite vertices, and is not degenerated into a square.  We conclude that that  the new hexagon is the regular one.  
Part (ii)  of the statement follows from part (i) just proved,  and from the known fact that 
the nonlocal perimeter decreases under Steiner symmetrization:
if $E\subseteq\R^2$  is a measurable set, and $E^*$ its Steiner symmetrization with respect to a given  axis, then 
$\Per_K(E)\geq \Per_K(E^*)$. Moreover, if $K$ is strictly decreasing, equality holds if and only if $E$ is a translate of $E^*$. We refer to \cite[Corollary 3.2, Theorem 3.3 and Corollary 3.4]{k}.  
  \end{proof}

 \bigskip 
 
 \begin{figure} [h] 
\centering   
\def\svgwidth{7cm}   
\begingroup%
  \makeatletter%
  \providecommand\color[2][]{%
    \errmessage{(Inkscape) Color is used for the text in Inkscape, but the package 'color.sty' is not loaded}%
    \renewcommand\color[2][]{}%
  }%
  \providecommand\transparent[1]{%
    \errmessage{(Inkscape) Transparency is used (non-zero) for the text in Inkscape, but the package 'transparent.sty' is not loaded}%
    \renewcommand\transparent[1]{}%
  }%
  \providecommand\rotatebox[2]{#2}%
  \ifx\svgwidth\undefined%
    \setlength{\unitlength}{265.53bp}%
    \ifx\svgscale\undefined%
      \relax%
    \else%
      \setlength{\unitlength}{\unitlength * \real{\svgscale}}%
    \fi%
  \else%
    \setlength{\unitlength}{\svgwidth}%
  \fi%
  \global\let\svgwidth\undefined%
  \global\let\svgscale\undefined%
  \makeatother%
  \begin{picture}(3.5,0.5)%
    \put(0.02,0){\includegraphics[height=4cm]{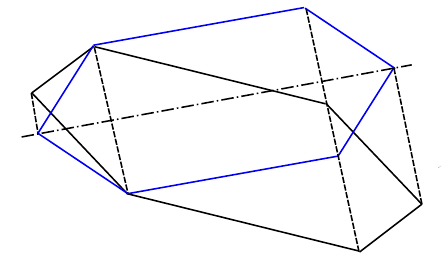} 
    \hskip 1.5 cm \includegraphics[height=4cm]{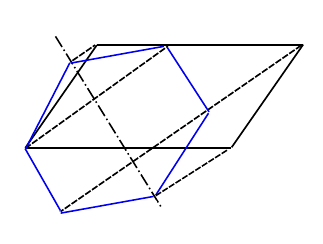}
    } %
    \put(0.256, 0.1){\color[rgb]{0,0,0}\makebox(0,0)[lb]{\smash{$A$}}}
    \put(0.8, -0.01){\color[rgb]{0,0,0}\makebox(0,0)[lb]{\smash{$B$}}}
        \put(0.93, 0.09){\color[rgb]{0,0,0}\makebox(0,0)[lb]{\smash{$C$}}}
    \put(0.74, 0.34){\color[rgb]{0,0,0}\makebox(0,0)[lb]{\smash{$D$}}}
        \put(0.17, 0.47){\color[rgb]{0,0,0}\makebox(0,0)[lb]{\smash{$E$}}}
    \put(0.04, 0.36){\color[rgb]{0,0,0}\makebox(0,0)[lb]{\smash{$F$}}}
    \put(0.76, 0.2){\color[rgb]{0,0,0}\makebox(0,0)[lb]{\smash{$B'$}}}
        \put(0.89, 0.38){\color[rgb]{0,0,0}\makebox(0,0)[lb]{\smash{$C'$}}}
    \put(0.7, 0.55){\color[rgb]{0,0,0}\makebox(0,0)[lb]{\smash{$D'$}}}
        \put(0.035, 0.28){\color[rgb]{0,0,0}\makebox(0,0)[lb]{\smash{$F'$}}}
    \put(1.23, 0.175){\color[rgb]{0,0,0}\makebox(0,0)[lb]{\smash{$A$}}}
    \put(1.56, 0.175){\color[rgb]{0,0,0}\makebox(0,0)[lb]{\smash{$B$}}}
    \put(1.76, 0.175){\color[rgb]{0,0,0}\makebox(0,0)[lb]{\smash{$C$}}}
    \put(1.93, 0.48){\color[rgb]{0,0,0}\makebox(0,0)[lb]{\smash{$D$}}}
    \put(1.62, 0.48){\color[rgb]{0,0,0}\makebox(0,0)[lb]{\smash{$E$}}}
    \put(1.44, 0.48){\color[rgb]{0,0,0}\makebox(0,0)[lb]{\smash{$F$}}}
    \put(1.29, 0.04){\color[rgb]{0,0,0}\makebox(0,0)[lb]{\smash{$B'$}}}
    \put(1.61, 0.08){\color[rgb]{0,0,0}\makebox(0,0)[lb]{\smash{$C'$}}}
    \put(1.72, 0.28){\color[rgb]{0,0,0}\makebox(0,0)[lb]{\smash{$D'$}}}
   \put(1.37, 0.465){\color[rgb]{0,0,0}\makebox(0,0)[lb]{\smash{$F'$}}}
 \end{picture}%
\endgroup%
\vskip .5 cm
\caption{The first symmetrization in the proof of of Lemma \ref{lemmahex}, in case of a hexagon (left) and of a parallelogram (right).}
\label{figura}   
\end{figure}
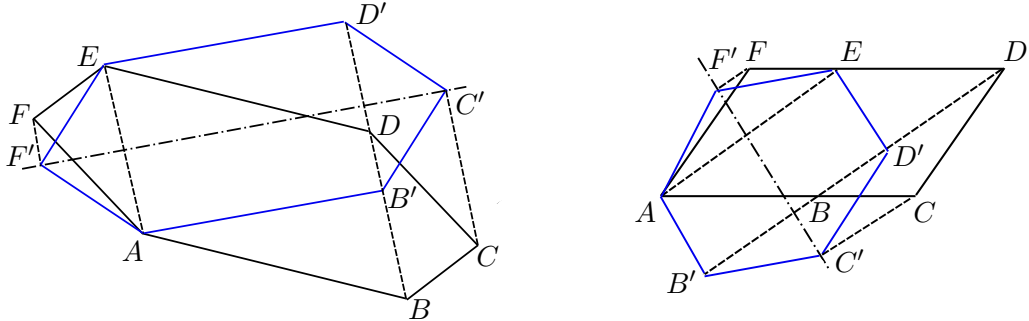

As an immediate consequence of Lemma \ref{lemmahex}, we obtain: 

 \begin{theorem}\label{coresa}  
 For any nonnegative symmetric  kernel $K$ on $\R ^2$,   satisfying  \eqref{frac} or \eqref{int},   
 the regular hexagon minimizes the nonlocal perimeter $\Per_K$ among convex polygons of the same area which tessellate $\R^2$ by translations.
 \end{theorem}

It is an open and interesting question  whether  the above result still holds removing the restriction to convex polygons.

 \bigskip\bigskip 
 
\noindent{\it Acknowledgements.} The authors are members of INDAM-GNAMPA; the third author was supported by the PRIN Project 2019/24. 



\begin{bibdiv}
\begin{biblist}
  	
\bib{BBFV}{article}{
    AUTHOR = {Bucur, D.},
    AUTHOR = {Fragal\`a, I.}, 
    AUTHOR = {Velichkov, B.},
    AUTHOR = {Verzini, G.},
     TITLE = {On the honeycomb conjecture for a class of minimal convex partitions 
},
   JOURNAL = {Trans. Amer. Math. Soc.}
   VOLUME = {370},
      YEAR = {2018},
    NUMBER = {10},
     PAGES = {7149--7179},
   
   }
   
   
   
\bib{BF1}{article}{
    AUTHOR = {Bucur, D.},
    AUTHOR = {Fragal\`a, I.}, 
     TITLE = {Proof of the honeycomb asymptotics for optimal Cheeger clusters
},
   JOURNAL = {Adv. Math.}
   VOLUME = {350},
      YEAR = {2019},
     PAGES = {97--129},
   
   }
   
   
      
\bib{BF2}{article}{
    AUTHOR = {Bucur, D.},
    AUTHOR = {Fragal\`a, I.}, 
     TITLE = {On the honeycomb conjecture for Robin Laplacian eigenvalues
},
   JOURNAL = {Commun. Contemp. Math.}
   VOLUME = {21},
      YEAR = {2019},
    NUMBER = {2},
     PAGES = {29 pp.},
   
   }

   
   
   
\bib{BBF}{article}{
    AUTHOR = {Bogosel, B.},
    AUTHOR = {Bucur, D.},
    AUTHOR = {Fragal\`a, I.},
     TITLE = {The nonlocal isoperimetric problem for polygons:  
Hardy-Littlewood and Riesz inequalities 
},
   JOURNAL = {Math. Ann.}
      YEAR = {2023},

   }
   
   
    
\bib{BCT}{article}{
    AUTHOR = {Bonacini, M.},
AUTHOR = {Cristoferi, R.},
AUTHOR = {Topaloglu, I.},
     TITLE = {Riesz-type inequalities and overdetermined problems for
              triangles and quadrilaterals},
   JOURNAL = {J. Geom. Anal.},
    VOLUME = {32},
      YEAR = {2022},
    NUMBER = {2},
     PAGES = {Paper No. 48},
     }
	
\bib{cassels}{book}{ 
    AUTHOR = {Cassels, J. W. S.},
     TITLE = {An introduction to the geometry of numbers},
    SERIES = {Classics in Mathematics},
      NOTE = {Corrected reprint of the 1971 edition},
 PUBLISHER = {Springer-Verlag, Berlin},
      YEAR = {1997},
     PAGES = {viii+344},
      ISBN = {3-540-61788-4},}
\bib{cn}{article}{
 AUTHOR = {Cesaroni, A.},
    author={Novaga, M.},
    TITLE = {The isoperimetric problem for nonlocal perimeters},
   JOURNAL = {Discrete Contin. Dyn. Syst. Ser. S},
    VOLUME = {11},
      YEAR = {2018},
    NUMBER = {3},
     PAGES = {425--440},	}
     
\bib{cnparti}{article}{
    AUTHOR = {Cesaroni, A.},
    author={Novaga, M.},
     TITLE = {Periodic partitions with minimal perimeter},
   JOURNAL = {ArXiv Preprint 2212.11545},
      YEAR = {2022},
   }
 	
\bib{cntiling}{article}{
    AUTHOR = {Cesaroni, A.},
    author={Novaga, M.},
     TITLE = {Minimal periodic foams with equal cells},
   JOURNAL = {ArXiv Preprint  2302.07112},
      YEAR = {2023},
   }
\bib{cdnp}{article}{
    AUTHOR = {Cicalese, M.},
    author={De Luca, L.},
    author={Novaga, M.},
    author={Ponsiglione,
              M.},
     TITLE = {Ground states of a two phase model with cross and self
              attractive interactions},
   JOURNAL = {SIAM J. Math. Anal.},
    VOLUME = {48},
      YEAR = {2016},
    NUMBER = {5},
     PAGES = {3412--3443},}
		         
   
\bib{choe}{article}{
   AUTHOR = {Choe, J.},
     TITLE = {On the existence and regularity of fundamental domains with
              least boundary area},
   JOURNAL = {J. Differential Geom.},
    VOLUME = {29},
      YEAR = {1989},
    NUMBER = {3},
     PAGES = {623--663},
}

%
 \bib{colombomaggi} {article}{
    AUTHOR = {Colombo, M.},
    author={Maggi, F.},
     TITLE = {Existence and almost everywhere regularity of isoperimetric
              clusters for fractional perimeters},
   JOURNAL = {Nonlinear Anal.},
       VOLUME = {153},
      YEAR = {2017},
     PAGES = {243--274},
}
 
\bib{CS}{book}{
author={J.H. Conway},
author={N.J.A. Sloane}, 
 TITLE = {Sphere packings, lattices and groups. Third edition},
    SERIES = {Grundlehren der mathematischen Wissenschaften},
    VOLUME = {290},
 PUBLISHER = {Springer-Verlag, New York},
      YEAR = {1999},
}
 
\bib{fedorov}{article}{ 
AUTHOR = {Fedorov, E.S.},
TITLE = {The Symmetry of Regular Systems of Figures}, 
JOURNAL = {Proceedings of the Imperial St. Petersburg Mineralogical Society}, 
VOLUME = {28}, 
YEAR = {1891},
PAGES = {1--146},
} 
\bib{Gal14}{article}{ 
AUTHOR = {Gallagher, Paul},
 author={Ghang, Whan},  
 author={Hu, David},
 author={Martin, Zane}, 
 author={Miller, Maggie},
 author={Perpetua, Byron}, 
 author={Waruhiu, Steven},
 title={Surface-area-minimizing n-hedral Tiles},
journal={Rose-Hulman Undergraduate Mathematics Journal}, 
volume={15},
year={2014},
NUMBER = {1},
PAGES = {Article 13},
}

\bib{hales}{article}{
author={ Hales, T.C.},
title={The honeycomb conjecture},
journal={Discrete Comput. Geom.},
volume={ 25},
 YEAR = {2001},
    NUMBER = {1},
     PAGES = {1--22},
}

\bib{JW}{article}{
author={ Jarohs, S.},
author= { Weth, T.}
title={Local compactness and nonvanishing for weakly singular
              nonlocal quadratic forms},
journal={Nonlinear Anal.},
volume={193},
 YEAR = {2020},
     PAGES = {111431, 15},
}





\bib{k}{article}{ 
    AUTHOR = {Kreuml, A.},
     TITLE = {The anisotropic fractional isoperimetric problem with respect
              to unconditional unit balls},
   JOURNAL = {Commun. Pure Appl. Anal.},
    VOLUME = {20},
      YEAR = {2021},
    NUMBER = {2},
     PAGES = {783--799},
     }

\bib{maggi}{book}{
    AUTHOR = {Maggi, F.},
     TITLE = {Sets of finite perimeter and geometric variational problems},
    SERIES = {Cambridge Studies in Advanced Mathematics},
    VOLUME = {135},
 PUBLISHER = {Cambridge University Press, Cambridge},
      YEAR = {2012},
}

\bib{ma}{article}{
AUTHOR = {Mahler, K.},
     TITLE = {On lattice points in {$n$}-dimensional star bodies. {I}.
              {E}xistence theorems},
   JOURNAL = {Proc. Roy. Soc. London Ser. A},
     VOLUME = {187},
      YEAR = {1946},
     PAGES = {151--187},
      }		
      
\bib{mnpr}{article}{
    AUTHOR = {Martelli, B.}, 
    AUTHOR = {Novaga, M.}, 
    AUTHOR = {Pluda, A.},
    AUTHOR = {Riolo, S.},
     TITLE = {Spines of minimal length},
   JOURNAL = {Ann. Sc. Norm. Super. Pisa Cl. Sci. (5)},
    VOLUME = {17},
      YEAR = {2017},
    NUMBER = {3},
     PAGES = {1067--1090},
     }
\bib{mm}{article}{
    AUTHOR = {McMullen, P.},
     TITLE = {Convex bodies which tile space by translation},
   JOURNAL = {Mathematika},
    VOLUME = {27},
      YEAR = {1980},
    NUMBER = {1},
     PAGES = {113--121},
}
	

\bib{morgansolo}{article}{
    AUTHOR = {Morgan, F.},
     TITLE = {The hexagonal honeycomb conjecture},
   JOURNAL = {Trans. Amer. Math. Soc. },
    VOLUME = {351},
      YEAR = {1999},
    NUMBER = {5},
     PAGES = {1753--1763},
}
		
\bib{npst}{article}{
    AUTHOR = {Novaga, M.}, 
    author={Paolini, E.}, 
    author={Stepanov, E.}, 
    author={ Tortorelli, V.M.},
     TITLE = {Isoperimetric clusters in homogeneous spaces via concentration compactness},
   JOURNAL = { J. Geom. Anal. },
    VOLUME = {32},
      YEAR = {2022},
      NUMBER = {11},
      PAGES = {Paper No. 263}, 
}

%

\bib{thompson}{article}{ 
AUTHOR = {Thomson (Lord Kelvin), W.},
     TITLE = {On the division of space with minimum partitional area},
   JOURNAL = {Acta Math.},
    VOLUME = {11},
      YEAR = {1887},
    NUMBER = {1-4},
     PAGES = {121--134},
}
%
\bib{kelvin}{book}{
EDITOR = {Weaire, D.},
     TITLE = {The {K}elvin problem},
         NOTE = {Foam structures of minimal surface area},
 PUBLISHER = {Taylor \& Francis, London},
      YEAR = {1996},
      }
 
%
\end{biblist}\end{bibdiv}
 \end{document}